\documentclass[11pt]{amsart}
\usepackage{amsmath,amsfonts,amssymb,amsthm,aliascnt}
\usepackage{latexsym,bm,graphicx}
\usepackage{mathrsfs}
\usepackage{color}
\definecolor{linkred}{rgb}{0.6,0.25,0.45}

\usepackage[
linktocpage,
colorlinks=true,
citecolor=linkred,
linkcolor=blue,
urlcolor=linkred,
backref=page
]{hyperref}
\usepackage[all]{xy}
\usepackage{enumitem}

\newcommand{\customitemize}[1]{%
	\begin{enumerate}[label={(#1\arabic*)}, ref=#1\arabic*]
	}

\title[Fundamental groups of $5$-manifolds]{Fundamental Groups in the Five Dimensional Pan--Rong Conjecture}
\author{Hongzhi Huang}
\address{Hongzhi Huang \\ Department of Mathematics \\ Jinan University\\ Guangzhou 510632}
\email{\href{mailto:huanghz@jnu.edu.cn}{huanghz@jnu.edu.cn}
}

\newtheorem{thm}{Theorem}[section]
\newtheorem{prop}[thm]{Proposition}
\newtheorem{lem}[thm]{Lemma}
\newtheorem{slem}[thm]{Sublemma}

\newtheorem{cor}[thm]{Corollary}

\newtheorem{claim}[thm]{Claim}

\newtheorem{mainthm}{Theorem}[section] 
\newaliascnt{myMainThm}{mainthm}        
\newtheorem*{myMainThm*}{Theorem \Alph{mainthm}} 

\theoremstyle{definition}
\theoremstyle{remark}

\newtheorem{defn}[thm]{Definition}
\newtheorem{rem}[thm]{Remark}

\numberwithin{equation}{section}

\newcommand {\diam }{\mathrm{diam}}
\newcommand {\Isom }{\mathrm{Isom}}

\newcommand {\vol }{\mathrm{vol}}

\newcommand {\id}{\mathrm{id}}

\newcommand {\Z}{\mathbb{Z}}

\newcommand {\R}{\mathbb{R}}

\newcommand {\GH}{\xrightarrow{GH}}

\newcommand {\Q}{\mathbb{Q}}

\newcommand {\spa}[1]{\langle{#1}\rangle}

\newcommand {\Ric}{\mathrm{Ric}}

\newcommand{\XXint}[3]{{
		\setbox0=\hbox{$#1{#2#3}{\int}$}
		\vcenter{\hbox{$#2#3$}}\kern-.5\wd0}}

\usepackage{etoolbox}
\makeatletter

\patchcmd{\@tocline}{\hfil}{\dotfill}{}{}

\patchcmd{\@startsection}
{\@sect{#1}{\@m}}
{\ifstrequal{#1}{subsection}
	{\@sect{#1}{\@M}}
	{\@sect{#1}{\@m}}}
{}{}

\makeatother

\date{\today}

\begin{document}

\maketitle
\begin{abstract}
Let $M$ be a complete open $5$-manifold with nonnegative Ricci curvature. If its universal cover $\tilde M$ has Euclidean volume growth, then $\pi_1(M)$ is finitely generated and virtually $\Z^k$ for some $0\le k\le4$. This confirms a conjecture of Pan and Rong \cite{PR18} in dimension five; see also \cite{BNS25,BrNaSe25}. In dimension six, under the same volume growth assumption, we also obtain a quantitative virtual abelianness result without assuming finite generation of the fundamental group.

\vspace*{10pt}
	
	\vspace*{10pt}

\end{abstract}

\setcounter{secnumdepth}{2}
\tableofcontents

\section{Introduction}

An open (complete, noncompact) $n$-manifold $M$ with $\Ric\ge0$ is said to have Euclidean volume growth if there exist $\theta\in(0,1]$ and $p\in M$ such that
\begin{equation}\label{eq:v}
	\vol(B_r(p))\ge\theta\,\omega_n\,r^n\qquad\text{for all }r>0,
\end{equation}
where $\omega_n$ is the volume of the unit ball in $\R^n$. In \cite{PR18}, Pan-Rong conjectured that, for $n\ge2$, the fundamental group of an open $n$-manifold with $\Ric\ge0$ whose universal cover has Euclidean volume growth is finitely generated (see also \cite{BNS25,BrNaSe25}). Without the volume growth assumption \eqref{eq:v}, the conclusion holds for $n=2,3$ \cite{CV35,Liu13} but fails for $n\ge4$ \cite{BNS25,BrNaSe25,WuYan26}. Under this assumption, it was confirmed in dimension four in \cite{HH25a}, while the general case remains open. This paper is a continuation of \cite{HH25a}, and here we give an affirmative answer in dimension five.
\begin{thm}\label{thm:main}
	Let $M$ be an open $5$-manifold with $\Ric\ge0$. If the universal cover $\tilde M$ has Euclidean volume growth, then the fundamental group $\pi_1(M)$ is finitely generated and virtually $\Z^k$ for some $k\in\{0,1,2,3,4\}$.
\end{thm}

The argument also applies in higher dimensions when every asymptotic cone of $M$ splits off an $\R^{n-5}$-factor.

\begin{thm}\label{thm:main-codim5}
	Let $M$ be an open $n$-manifold with $\Ric\ge0$, where $n\ge5$. If the universal cover $\tilde M$ has Euclidean volume growth and every asymptotic cone of $M$ splits off an $\R^{n-5}$-factor isometrically, then
	\begin{itemize}
		\item [(1)] $\pi_1(M)$ is finitely generated and virtually $\Z^k$ for some $k\in\{0,1,\ldots,5\}$;
		\item [(2)] every asymptotic cone of $M$ has a pole at its reference point;
		\item [(3)] the escape rate of $\pi_1(M)$ is zero.
	\end{itemize}
\end{thm}

The pole condition in (2) implies finite generation by \cite{Sor99}, while the vanishing escape rate in (3) implies virtual abelianness by \cite{Pan21}. Our proof establishes (1) first and then derives (2) and (3).

Although we cannot prove Pan-Rong's conjecture in dimension $6$, we obtain a quantitative virtual-abelianness result.

\begin{thm}\label{thm:main-codim6}
	For $n\ge 2,\theta\in(0,1]$, there exists a constant $C(n,\theta)>0$ to the following effect. Let $M$ be an open $n$-manifold with $\Ric\ge0$, $n\ge6$. If the universal cover $\tilde M$ has Euclidean volume growth with the constant $\theta$ in \eqref{eq:v} and every asymptotic cone of $M$ splits off an $\R^{n-6}$-factor isometrically, then $\pi_1(M)$ contains an abelian subgroup of index bounded by $C(n,\theta)$.
\end{thm}

\begin{rem}
We conjecture that this constant should in fact be universal, depending only on $n$ and not on $\theta$.
\end{rem}

\subsection*{Outline of the proof}

We now describe the proof of the finite generation part of Theorem \ref{thm:main}. This paper adopts an idea similar to that of \cite{HH25a}. Let $\Omega:=\Omega(\tilde M,\pi_1(M))$ denote the collection of equivariant asymptotic cones of $(\tilde M,\pi_1(M))$. By Pan's observation (Lemma \ref{lem:short-gen-cone}), it suffices to rule out the case where the central orbit $G\cdot y$ is disconnected. The benefits of the low dimension and of the Euclidean volume growth of the universal cover enable us to restrict the equivariant asymptotic cones that can occur to very few possibilities. A careful analysis of the isotropy of the limit groups acting on the cross-section of an asymptotic cone then rules out the bad behaviours of $G\cdot y$.

More specifically, suppose that $\pi_1(M)$ is not finitely generated. By Wilking's reduction (Lemma \ref{lem:wilking}), after passing to a cover of $M$, we may assume that $\pi_1(M)$ is abelian. For $(Y,y,G)\in\Omega$, we call $I(G):=G_y$ the central isotropy group. We first prove the following central isotropy criterion (Proposition \ref{prop:isotropy-criteria}): if either some cone in $\Omega$ has finite central isotropy, or the dimension of the central isotropy is constant over $\Omega$, then all central isotropy groups arising in $\Omega$ are isomorphic, and consequently $\pi_1(M)$ is finitely generated. This is a refined version of \cite[Lemma 4.1]{HH25a}, obtained via the critical scaling argument developed by Pan \cite{Pan19}. Then we split the discussion into two cases, according as $\pi_1(M)$ is torsion-free or torsion. 

In the torsion-free case, every asymptotic cone is isometric to $\R\times C(Z)$ for some $3$-dimensional $\mathrm{RCD}(2,3)$-space $Z$, and the central isotropy acts effectively on $Z$. By \cite{BPS24}, such $Z$ must be homeomorphic to a spherical $3$-manifold. Using the central isotropy criterion together with the abelianness assumption, we conclude that $\dim I(G)\in\{1,2\}$. It remains to show that $\dim I(G)=1$ is the only possibility; this is achieved by examining effective $T^k$-actions ($k\in\{1,2\}$) on $3$-dimensional spherical manifolds, and then applying a critical-scaling argument.

In the torsion case, a key tool is Lemma \ref{lem:noncodim3} (\cite[Lemma 3.2]{HH25a}), which roughly says the following: if a noncollapsed equivariant convergence involves finite normal covering groups whose orbits through the reference point shrink to $0$, and if the limit is a metric cone splitting off an $\R^{n-3}$-factor, then these groups are trivial for all sufficiently large indices. Under the contrary assumption, it is not difficult to see that there exists $(Y,y,G)\in\Omega$ such that $Y\cong\R\times C(Z)$, where $Z$ is a $3$-dimensional spherical manifold and $I(G)$ fixes the $\R$-factor and acts effectively on $Z$. The remaining key idea is then to find $z_0\in Z$ and a nontrivial $\gamma_\infty\in I(G)$ such that (i) $\gamma_\infty$ fixes $z_0$, and hence fixes a half-plane of $Y$; (ii) for some $\gamma\in\pi_1(M)$ and some sequence $r_i\to0$, $(r_i\tilde M,\tilde p,\spa\gamma)\GH(Y,y,\spa{\gamma_\infty})$. Once such elements are found, blow up the sequence slowly so that it converges to the tangent cone $T_{(0,1,z_0)}Y$ of $Y$ at $(0,1,z_0)$; Lemma \ref{lem:noncodim3} then yields a contradiction. To find such $z_0$ and $\gamma_\infty$, we analyze the action of the identity component $I(G)_0$ on $Z$, where the orientability of noncollapsed Ricci limit spaces plays an indispensable role \cite{BrBrPi24}.

\subsection*{AI disclosure}
This work used AI assistance. The main approach is based on \cite{HH25a}, and several important steps were resolved with the help of AI. The author wrote the initial draft of the entire paper and subsequently used AI to improve the wording. The author is responsible for the correctness of all content.

\section{Preliminaries}

\subsection{Conventions and basic lemmas}

In this article, for an open manifold $N$ with $\Ric\ge0$ and a closed subgroup $\Gamma<\Isom(N)$, we denote by $\Omega(N)$ and $\Omega(N,\Gamma)$ the collections of the asymptotic cones and the equivariant asymptotic cones of $N$, respectively. We always write $(\tilde N,\tilde x)$ for the universal cover of $(N,x)$, where $x$ is the projection of $\tilde x$. If $N$ has Euclidean volume growth, then for every $(Y,y,G)\in\Omega(N,\Gamma)$, $Y$ is a metric cone with vertex $y$ \cite{CC96} and $G$ is a Lie group \cite{CN12}. Write the maximal Euclidean decomposition of $(Y,y)$ as
$$(Y,y)\cong(\R^k\times C(Z),(0^k,o_Z)),$$
where $C(Z)$ is the metric cone over a length metric space $Z$ with $\diam Z<\pi$. Note that $\Isom(Y)\cong\Isom(\R^k)\times\Isom(Z)$. The orbit $G\cdot y$ is contained in $\R^k\times\{o_Z\}$, and we write $\mathrm{span}_\R(G\cdot y)$ for the linear subspace spanned by $G\cdot y$ in $\R^k\times\{o_Z\}$.

In what follows, we list some standard lemmas, stated in a form convenient for our use.

\begin{lem}\label{lem:freesplits}
	Let $N$ be an open manifold with $\Ric\ge0$ and Euclidean volume growth. If $\Isom(N)$ contains a closed subgroup $\Gamma\cong\Z^k$, then for every $(Y,y,G)\in\Omega(N,\Gamma)$, $G$ contains a closed $\R^k$-subgroup and $Y$ splits off an $\R^k$-factor isometrically.
\end{lem}

\begin{proof}
	By \cite[Corollary 3.2]{PY24}, $G$ contains a closed $\R^k$-subgroup. Every isotropy group of $G$ is compact, so this $\R^k$-subgroup acts freely. In particular, the orbit $\R^k\cdot y$ is homeomorphic to $\R^k$, and the conclusion follows from the discussion preceding the lemma.
\end{proof}

\begin{rem}\label{rem:orbit-not-standard}
	In general, the orbit of the $\R^k$-subgroup at $y$ need not be isometric to the standard metric $\R^k$. Thus $\dim\mathrm{span}_\R(\R^k\cdot y)\ge k$, and strict inequality may occur.
\end{rem}

\begin{lem}\label{lem:wilking}
	If $N$ is open with $\Ric\ge0$ and $\pi_1(N)$ is not finitely generated, then $\pi_1(N)$ contains an abelian subgroup that is not finitely generated. Moreover, such a subgroup may be chosen to be either torsion or torsion-free.
\end{lem}

\begin{proof}
	The first statement is Wilking's reduction \cite{Wil00}. For the second, let $\Gamma$ be an abelian subgroup of $\pi_1(N)$ that is not finitely generated, and let $T$ be its torsion subgroup. If $T$ is not finitely generated, we are done. Thus assume that $T$ is finitely generated, in which case $T$ is finite. Put $k:=\#T$, and define $\iota\colon\Gamma\to\Gamma$ by $\iota(\gamma)=\gamma^k$. Since $\Gamma$ is abelian, $\iota$ is a homomorphism with kernel $T$. Hence $H:=\iota(\Gamma)\cong\Gamma/T$ is torsion-free; moreover, as $T$ is finite and $\Gamma$ is not finitely generated, $H$ is not finitely generated.
\end{proof}

The following observation of Pan \cite{Pan20} provides a sufficient condition for the fundamental group to be finitely generated.
\begin{lem}\label{lem:short-gen-cone}
	If $N$ is an open manifold with $\Ric\ge0$ and $\pi_1(N)$ is not finitely generated, then there exists $(Y,y,G)\in\Omega(\tilde N,\pi_1(N))$ whose orbit $G\cdot y$ is disconnected.
\end{lem}

\subsection{Critical scaling principle}

The critical scaling argument, developed by Pan \cite{Pan19,Pan25}, aims to show that certain types of equivariant asymptotic cones cannot coexist. If both types were present, one could choose a critical scale at which the behaviour changes; rescaling the limit at that scale would then contradict the way the scale was chosen. The argument uses more than connectedness. For instance, under the relevant compactness assumptions, the connected family consisting of $(\R,0)$, $\bigl([a,+\infty),0\bigr)$ for $a<0$, and $\bigl([0,+\infty),0\bigr)$ cannot be the full family of asymptotic cones of a fixed space. In this subsection, we abstract this argument into a principle that is more convenient to use.

A pointed equivariant proper metric space is the equivariant isometry class
of a triple $(X,x,G)$, where $(X,x)$ is a pointed proper metric space and
$G<\Isom(X)$ is closed. For $P=(X,x,G)$ and $a>0$, let
$aP:=(aX,x,G)$. For a nonempty collection $\mathcal A$ of pointed
equivariant proper metric spaces, write
$$
d_{GH}(P,\mathcal A)
:=\inf_{A\in\mathcal A}d_{GH}(P,A).
$$
All closures below are taken in $\Omega(N,\Gamma)$.

\begin{prop}[Critical-scaling principle]
	\label{prop:critical-scaling}
	Let $(N,p)$ be an open manifold with $\Ric\geq0$, and let
	$\Gamma<\Isom(N)$ be closed. Let $\mathcal U$ be a nonempty open subset
	of $\Omega(N,\Gamma)$. Suppose that there are
	$\sigma\in\{1,-1\}$ and $R\in(1,+\infty]$ such that, for every
	$Q\in\overline{\mathcal U}$, there is a $c\in(1,R]$ satisfying
	$c^{-\sigma}Q\in\mathcal U$. Then, for every
	$B\in\Omega(N,\Gamma)$, there exists $a\in[1,R)$ such that $a^\sigma B\in\overline{\mathcal U}$.
\end{prop}

\begin{rem}
	When $R=+\infty$, we allow $c=+\infty$ in the proposition above. If
	$\sigma=1$, the condition $c^{-1}Q\in\mathcal U$ for $c=+\infty$ means
	that some equivariant asymptotic cone of $Q$ belongs to $\mathcal U$.
	If $\sigma=-1$, the condition $cQ\in\mathcal U$ for $c=+\infty$ means
	that some equivariant tangent cone of $Q$ at the reference point belongs
	to $\mathcal U$.
\end{rem}
\begin{proof}
	Suppose, to the contrary, that there is a $B\in\Omega(N,\Gamma)$ such
	that $a^\sigma B\notin\overline{\mathcal U}$ for every $a\in[1,R)$.
	Put $\mathcal C:=\Omega(N,\Gamma)\setminus\mathcal U$. Since
	$B\notin\overline{\mathcal U}$, the set $\mathcal C$ is nonempty.

	We first note that if $P_i\GH Q\in\Omega(N,\Gamma)$ and $d_{GH}(P_i,\mathcal U)
	<d_{GH}(P_i,\mathcal C)$, then $Q\in\overline{\mathcal U}$. Indeed, passing to the limit gives
	$d_{GH}(Q,\mathcal U)\leq d_{GH}(Q,\mathcal C)$,
	whereas $Q\notin\overline{\mathcal U}$ would imply
	$d_{GH}(Q,\mathcal U)>0=d_{GH}(Q,\mathcal C)$.

	We may always choose the return factor $c$ in $(1,R)$. If
	$R<+\infty$ and $c=R$, this follows from the continuity of rescaling
	and the openness of $\mathcal U$. If $R=+\infty$ and $c=+\infty$, it
	follows from the definition of an equivariant tangent or asymptotic
	cone and the openness of $\mathcal U$.

	Fix $A\in\mathcal U$.

	\textbf{Case 1:} $\sigma=1$.
	Interlacing the defining scales of $A$ and $B$, choose $r_i,s_i\to0$
	such that
	$$
	(s_iN,p,\Gamma)\GH B,\qquad
	(r_iN,p,\Gamma)\GH A,\qquad
	\frac{r_i}{s_i}\to\infty.
	$$
	Put $N_i=s_iN$ and $\Lambda_i:=r_i/s_i$. Then
	$$
	(N_i,p,\Gamma)\GH B,\qquad
	(\Lambda_iN_i,p,\Gamma)\GH A,\qquad
	\Lambda_i\to\infty.
	$$
	Set
	$$
	E_i:=\left\{t\in[1,\Lambda_i]\ \middle|\
	d_{GH}((tN_i,p,\Gamma),\mathcal U)
	<d_{GH}((tN_i,p,\Gamma),\mathcal C)\right\}.
	$$
	Since $(\Lambda_iN_i,p,\Gamma)\GH A\in\mathcal U$, the set $E_i$ is
	nonempty for large $i$. Choose $\lambda_i\in E_i$ such that
	$\inf E_i\leq\lambda_i<e^{1/i}\inf E_i$.

	We claim that $\liminf\lambda_i\geq R$, where this means
	$\lambda_i\to+\infty$ if $R=+\infty$. Otherwise, after passing to a
	subsequence, $\lambda_i\to a\in[1,R)$, and hence
	$(\lambda_iN_i,p,\Gamma)\GH aB$. Since $\lambda_i\in E_i$, the preceding
	observation gives $aB\in\overline{\mathcal U}$, a contradiction.

	After passing to a subsequence,
	$(\lambda_iN_i,p,\Gamma)\GH Q$. Since
	$\lambda_is_i\leq r_i\to0$ and $\lambda_i\in E_i$, we have
	$Q\in\overline{\mathcal U}$. Choose $c\in(1,R)$ such that
	$c^{-1}Q\in\mathcal U$. Since $\liminf\lambda_i\geq R>c$, for large $i$
	we have $c^{-1}\lambda_i\in[1,\Lambda_i]$. The convergence
	$(c^{-1}\lambda_iN_i,p,\Gamma)\GH c^{-1}Q\in\mathcal U$ shows that
	$c^{-1}\lambda_i\in E_i$ for large $i$. Therefore $\lambda_i<e^{1/i}\inf E_i\leq e^{1/i}c^{-1}\lambda_i$, which is impossible for large $i$.

	\textbf{Case 2:} $\sigma=-1$.
	Interlacing in the opposite order, choose $r_i,s_i\to0$ such that
	$$
	(r_iN,p,\Gamma)\GH A,\qquad
	(s_iN,p,\Gamma)\GH B,\qquad
	\frac{s_i}{r_i}\to\infty.
	$$
	Put $N_i=r_iN$ and $\Lambda_i:=s_i/r_i$. Then
	$$
	(N_i,p,\Gamma)\GH A,\qquad
	(\Lambda_iN_i,p,\Gamma)\GH B,\qquad
	\Lambda_i\to\infty.
	$$
	Define $E_i$ as above. Since $(N_i,p,\Gamma)\GH A\in\mathcal U$, the
	set $E_i$ is nonempty for large $i$. Choose $\lambda_i\in E_i$ such
	that $e^{-1/i}\sup E_i<\lambda_i\leq\sup E_i$.

	We claim that $\liminf(\Lambda_i/\lambda_i)\geq R$, where this means
	$\Lambda_i/\lambda_i\to+\infty$ if $R=+\infty$. Otherwise, after
	passing to a subsequence, $\Lambda_i/\lambda_i\to a\in[1,R)$, and hence
	$(\lambda_iN_i,p,\Gamma)\GH a^{-1}B$. Since $\lambda_i\in E_i$, the
	preceding observation gives $a^{-1}B\in\overline{\mathcal U}$, a
	contradiction.

	After passing to a subsequence,
	$(\lambda_iN_i,p,\Gamma)\GH Q$. Since
	$\lambda_ir_i\leq s_i\to0$ and $\lambda_i\in E_i$, we have
	$Q\in\overline{\mathcal U}$. Choose $c\in(1,R)$ such that
	$cQ\in\mathcal U$. Since
	$\liminf(\Lambda_i/\lambda_i)\geq R>c$, for large $i$ we have
	$c\lambda_i\in[1,\Lambda_i]$. The convergence
	$(c\lambda_iN_i,p,\Gamma)\GH cQ\in\mathcal U$ shows that
	$c\lambda_i\in E_i$ for large $i$. Therefore $c\lambda_i\leq\sup E_i<e^{1/i}\lambda_i$, which is impossible for large $i$.
\end{proof}

In the case $R=+\infty$, the following more convenient form holds.

\begin{cor}\label{cor:critical-scaling}
	Let $(N,p)$ be an open manifold with $\Ric\geq0$, and let
	$\Gamma<\Isom(N)$ be closed. Let $\mathcal U$ be a nonempty open subset
	of $\Omega(N,\Gamma)$.
	\begin{itemize}
		\item [(1)] Suppose that some equivariant asymptotic cone of every
		$Q\in\overline{\mathcal U}$ belongs to $\mathcal U$. Then, for every
		$B\in\Omega(N,\Gamma)$, there exists $a\in[1,+\infty)$ such that
		$aB\in\overline{\mathcal U}$. In particular, some equivariant asymptotic cone of $B$ belongs to $\mathcal U$.

		\item [(2)] Suppose that some equivariant tangent cone of every
		$Q\in\overline{\mathcal U}$ at the reference point belongs to
		$\mathcal U$. Then, for every $B\in\Omega(N,\Gamma)$, there exists
		$a\in[1,+\infty)$ such that $a^{-1}B\in\overline{\mathcal U}$. In particular, some equivariant tangent cone of $B$ at the reference point belongs to
		$\mathcal U$.
	\end{itemize}
\end{cor}

\section{Isotropy stability}
\label{sec:isotropy}

This section gives a stability condition for the isotropy at the reference point of every equivariant limit, which can be viewed as a refinement of \cite[Lemma 4.1]{HH25a}.

Throughout this section, let $(M,p)$ be an open $n$-manifold with $\Ric\ge0$ and Euclidean volume growth, and let $\Gamma$ be a closed abelian subgroup of $\Isom(M)$. Denote by $\Omega(M,\Gamma)$ the collection of all equivariant asymptotic cones of $(M,\Gamma)$. For $(Y,y,G)\in\Omega(M,\Gamma)$, write $I(G)$ for the isotropy subgroup of $G$ at $y$, called the central isotropy of $G$, and write $CS(Y)$ for the unit cross-section of $Y$ with its intrinsic metric. Then $I(G)$ is a closed subgroup of $\Isom(CS(Y))$, and by Colding--Naber \cite{CN12} both $G$ and $I(G)$ are Lie groups. Moreover, $CS(Y)$ is a noncollapsed $\mathrm{RCD}(n-2,n-1)$-space; see \cite{K15,DPG18}.

We use Pan's property $(P)$ \cite{Pan25}, in the formulation of \cite[Definition~4.2]{HH25a}.

\begin{defn}[Property $(P)$]\label{def:P}
	Let $(Y,y,G)\in\Omega(M,\Gamma)$, and write its maximal Euclidean
	decomposition as $Y=E\times C(Z)$. For $g\in G$, write $g(e,r,z)=(Q_ge+a_g,r,\sigma_gz)$. We say that $(Y,y,G)$ has property $(P)$ if both
	$$
	g_0(e,r,z):=(Q_ge,r,\sigma_gz)
	\quad\text{and}\quad
	T_g(e,r,z):=(e+a_g,r,z)
	$$
	belong to $G$ for every $g\in G$.
\end{defn}

\begin{rem}\label{rem:P}
	If $(Y,y,G)$ has property $(P)$, then abelian $G$ is isomorphic to the direct product of $I(G)$ and its translation subgroup. Moreover, the central orbit is a closed additive subgroup of $E$ and hence has the form
	\begin{equation}\label{eq:central-orbit}
		G\cdot y=V\oplus\Lambda,
	\end{equation}
	where $V<E$ is a vector subspace and $\Lambda$ is a lattice in a subspace of $V^\perp$.  In particular, $\Lambda\cong\mathbb Z^k$ for some $k$.
	
	Since $G\cdot y$ lies in a Euclidean factor of $Y$, we let $\mathrm{span}_\R(G\cdot y)$ denote the subspace spanned by the elements of $Gy$. Under dilation,
	$$
	r(G\cdot y)\GH\operatorname{span}_{\mathbb R}(G\cdot y)
	\quad\text{as }r\to0,
	\qquad
	r(G\cdot y)\GH V
	\quad\text{as }r\to\infty.
	$$
	The corresponding equivariant limits remain in $\Omega(M,\Gamma)$.
\end{rem}

In the following proof, $d_{GH}$ always denotes the pointed
Gromov--Hausdorff distance between the central orbits equipped with
their induced metrics. Thus, for example,
$d_{GH}(r(G\cdot y),\R^m)$ compares the pointed orbit
$G\cdot y$ in the rescaled space $rY$ with $(\R^m,0^m)$.

\begin{prop}
	\label{prop:Pimpliesconnectedorbit}
	If every element of $\Omega(M,\Gamma)$ has property $(P)$, then there is an integer $0\leq m\leq n$ such that the central orbit $G\cdot y$ is isometric to $\R^m$ for every $(Y,y,G)\in\Omega(M,\Gamma)$. 
\end{prop}

\begin{proof}
	
	By Remark~\ref{rem:P}, the tangent cone at the reference point and the asymptotic cone at infinity of any element of $\Omega(M,\Gamma)$ have Euclidean central orbits. Let $m$ be the maximal dimension of such an orbit, and call $(Y,y,G)\in\Omega(M,\Gamma)$ good if $G\cdot y\cong\R^m$.
	
	There exists $\epsilon\in(0,0.1^n)$ such that, for any $(X,x,H)\in\Omega(M,\Gamma)$, if $d_{GH}(H\cdot x,\R^m)<2\epsilon$, then $\dim\operatorname{span}_{\R}(H\cdot x)\ge m$. Since $\Omega(M,\Gamma)$ is closed under blowing down, and by the maximality of $m$, we also have $\dim\operatorname{span}_{\R}(H\cdot x)\le m$. Hence, for any $(X,x,H)\in\Omega(M,\Gamma)$, if $d_{GH}(H\cdot x,\R^m)<2\epsilon$, then
	\begin{equation}\label{eq:Hx}
		\dim\operatorname{span}_{\R}(H\cdot x)=m.
	\end{equation}
	
	Let $S$ be the collection of good elements of $\Omega(M,\Gamma)$, which is obviously nonempty, and set $\mathcal U:=B_\epsilon(S)$, the open $\epsilon$-neighborhood of $S$ in $\Omega(M,\Gamma)$. For every $(X,x,H)\in\overline{\mathcal U}$, by the choice of $\epsilon$, \eqref{eq:Hx} holds. Hence every asymptotic cone of $(X,x,H)$ is good, that is, belongs to $S$. Applying Corollary \ref{cor:critical-scaling}, we conclude that for any $(X,x,H)\in\Omega(M,\Gamma)$ there exists $a\in[1,+\infty)$ such that $a(X,x,H)\in\overline{\mathcal U}$. This implies that \eqref{eq:Hx} holds for every $(X,x,H)\in\Omega(M,\Gamma)$, and hence every such element is good; otherwise its tangent cone at the reference point would contradict (\ref{eq:Hx}). The proof is finished.
\end{proof}

We recall the following facts which will be used in the proof of the next proposition. For $(Y_i,y_i,G_i)\GH(Y,y,G)\in\Omega(M,\Gamma)$,

\customitemize{P}
	\item\label{itm:P1} Any equivariant tangent cone at $y$ has property $(P)$ and its isotropy is isomorphic to $I(G)$ \cite[Lemma~4.3]{HH25a};
	\item \label{itm:P2} Up to a subsequence, $(CS(Y_i),I(G_i))\GH(CS(Y),K)$ where $K\subset I(G)$, and for every large $i$, there is an $\epsilon_i$-Gromov-Hausdorff approximation $\phi_i:I(G_i)\to K$ which is a Lie group monomorphism (\cite{MRW08,PR18}), where $\epsilon_i\to0$.
	\item\label{itm:P3} If each $(Y_i,y_i,G_i)$ satisfies property $(P)$, then $(CS(Y_i),I(G_i))\GH(CS(Y),I(G))$ \cite[Lemma~4.4]{HH25a}.
\end{enumerate}

The following proposition follows by an argument analogous to that of \cite[Lemma 4.1]{HH25a}. In its proof, all isotropy groups mentioned are central isotropy.

\begin{prop}[Asymptotic-isotropy criterion]\label{prop:isotropy-criteria}
	Each of the following conditions implies that all central isotropy groups arising in $\Omega(M,\Gamma)$ are isomorphic.
	\begin{enumerate}[label=\textup{(\arabic*)}]
		\item one cone in $\Omega(M,\Gamma)$ has finite central isotropy;
		\item the integer $\dim I(G)$ is the same for every $(Y,y,G)\in\Omega(M,\Gamma)$.
	\end{enumerate}
	Consequently, every element in $\Omega(M,\Gamma)$ has property $(P)$. 
\end{prop}
\begin{rem}
	The stability of isotropies of this proposition does not rely on the abelianness of $\Gamma$, which therefore implies \cite[Lemma 4.1]{HH25a}. 
\end{rem}

\begin{proof}
	In case~(1), choose a finite isotropy $K$ with the smallest order. In case~(2), choose an isotropy $K$ with the smallest number of components among the groups of the common dimension. By (\ref{itm:P1}), we may replace its cone by an equivariant tangent cone at the reference point and assume that the cone has property $(P)$.

	There is an $\eta>0$ such that any asymptotic cone which is $\eta$-close to a property-$(P)$ cone with central isotropy $K$ also has isotropy isomorphic to $K$. Otherwise, take two sequences converging to the same limit. By (\ref{itm:P2}) and (\ref{itm:P3}), the isotropies of the property-$(P)$ sequence converge to the full limiting isotropy $J$ and admit monomorphisms into $J$. The minimal choice of $K$ implies $J\simeq K$. The isotropies of the second sequence also inject into $J$, and the same minimality shows that these monomorphisms are onto. This proves the gap. The same argument shows that the class of property-$(P)$ cones with isotropy $K$ is sequentially closed.

	Let $\Omega_K^P$ be the collection of $(Y,y,G)\in\Omega(M,\Gamma)$ satisfying $I(G)\cong K$ and property $(P)$, and define $\mathcal U:=B_{0.5\eta}(\Omega_K^P)$, where $\eta$ is as above. Combining the above discussion and \cite[Lemma 4.3]{HH25a}, one verifies directly that the condition (2) of Corollary \ref{cor:critical-scaling} applies to $\mathcal U$. Hence, for every $(Y,y,G)\in\Omega(M,\Gamma)$, there exists $a\in[1,+\infty)$ such that $a^{-1}(Y,y,G)\in\overline{\mathcal U}$. Applying the property of $\eta$ again, we get that $a^{-1}(Y,y,G)$ has isotropy $K$, and hence $I(G)\cong K$.

	Blow down any cone action about its reference point. The original isotropy embeds into the limiting isotropy again by (\ref{itm:P2}), and the uniqueness just proved shows that this embedding is onto. For $g(e,r,z)=(Q_ge+a_g,r,\sigma_gz)$, its blowdown has vertex-fixing part $g_0(e,r,z)=(Q_ge,r,\sigma_gz)\in I(G)$. Therefore $gg_0^{-1}$ is the pure translation by $a_g$, and the cone has property $(P)$.

\end{proof}

\section{Finite generation}

In this section, we prove the finite generation part of Theorem \ref{thm:main}. We first prove it assuming Theorem \ref{thm:free-abelian-fg} and Theorem \ref{thm:torsion}, which deal respectively with the torsion-free subgroup and the torsion subgroup; our main effort later goes into proving these two theorems.  The same argument applies under the hypotheses of Theorem \ref{thm:main-codim5}, after lifting the $\R^{n-5}$-factor and restricting attention to the factors on which the limiting group acts nontrivially. 

\begin{proof}[Proof of finite generation of Theorem \ref{thm:main}]
	To begin with, the hypotheses of Theorem \ref{thm:main} are preserved under passage to any cover. Arguing by contradiction, suppose that $\pi_1(M)$ is not finitely generated. By Lemma \ref{lem:wilking}, after passing to an appropriate cover of $M$, we may assume that $\pi_1(M)$ is abelian and not finitely generated, and that it is either torsion or torsion-free. We then obtain a contradiction from Theorem \ref{thm:free-abelian-fg} or Theorem \ref{thm:torsion}.
\end{proof}

\subsection{$\pi_1(M)$ is torsion-free}\label{sec:free-abelian}

In this section, we prove the following.
\begin{thm}\label{thm:free-abelian-fg}
	Let $M$ be as in Theorem \ref{thm:main}. Further assume that $\Gamma:=\pi_1(M)$ is nontrivial torsion-free abelian. Then there exists an integer $l\ge0$ such that, for every $(Y,y,G)\in\Omega(\tilde M,\Gamma)$, the central orbit $G\cdot y$ is isometric to $\R^l$. Consequently, $\pi_1(M)$ is finitely generated.
\end{thm}
\begin{rem}\label{rem:codim5}
	The conclusion also holds for $M$ as in Theorem \ref{thm:main-codim5}; since the proof presents no new difficulties, we do not write it out in this general case, so as to keep the notation simple. This fact will be used below to prove parts (1) and (2) of Theorem \ref{thm:main-codim5}.
\end{rem}

Now we start to prove Theorem \ref{thm:free-abelian-fg}. 

\begin{lem}\label{lem:dimIG}
	For every $(Y,y,G)\in\Omega(\tilde M,\Gamma)$, there exists an isometric decomposition $(Y,y)\cong(\R\times C(Z),(0,o_Z))$, where $Z$ is a spherical $3$-manifold, such that the following hold:
	\begin{itemize}
		\item [(1)] $G$ contains a closed $\R$-subgroup and $\R\times\{o_Z\}\subset\mathrm{span}_\R(G\cdot y)$;
		\item [(2)] $I(G)$ acts on the first $\R$-factor trivially;
		\item [(3)] $\dim I(G)\le2$.
	\end{itemize}
	Consequently, $Y$ is isometric to $C(S^0*Z)$, and $I(G)$ fixes $S^0$ and acts on $Z$ effectively, where $*$ denotes the spherical join.
\end{lem}

\begin{proof}
	Since $\Gamma$ contains a $\Z$, Lemma \ref{lem:freesplits} shows that $G$ contains a closed $\R$-subgroup and that $Y$ splits off an $\R$-factor isometrically. As $G$ is abelian, $I(G)$ fixes $y$, hence fixes every point of $G\cdot y$, and therefore fixes $\mathrm{span}_\R(G\cdot y)$ pointwise. Choose the decomposition $(Y,y)\cong(\R\times C(Z),(0,o_Z))$ so that $\R\times\{o_Z\}\subset\mathrm{span}_\R(G\cdot y)$. Then $I(G)$ fixes the $\R$-factor and acts effectively on $Z$. By \cite[Theorem 1.4]{BPS24}, $Z$ is a spherical $3$-manifold. Finally, since $I(G)$ is abelian, $\dim I(G)\ge3$ would force $I(G)$ to contain a subgroup isomorphic to $T^3$; however, a spherical $3$-manifold admits no effective $T^3$-action. Hence $\dim I(G)\le2$.
\end{proof}

In view of Propositions \ref{prop:isotropy-criteria} and \ref{prop:Pimpliesconnectedorbit} and Lemma \ref{lem:dimIG}, we may assume henceforth that $\dim I(G)\in\{1,2\}$. And since $\Omega(\tilde M,\Gamma)$ is connected in the equivariant Gromov--Hausdorff topology, Theorem \ref{thm:free-abelian-fg} will follow once we prove the following.

\begin{prop}\label{prop:free-abelian-fg}
	Let $M$ be as in Theorem \ref{thm:free-abelian-fg}. Assume further that $\dim I(G)\in\{1,2\}$ for every $(Y,y,G)\in\Omega(\tilde M,\Gamma)$. Then, for every $(Y,y,G)\in\Omega(\tilde M,\Gamma)$, the orbit $G\cdot y$ is isometric to $\R^l$ for some integer $l\ge0$, where $l$ may depend on $(Y,y,G)$.
\end{prop}

We argue by contradiction; assume that there exists $(Y_1,y_1,G_1)\in\Omega(\tilde M,\Gamma)$ such that $G_1\cdot y_1$ is not isometric to any Euclidean space.

\begin{claim}\label{claim:iso-fixed}
	Such a triple $(Y_1,y_1,G_1)$ satisfies $\dim I(G_1)=1$, and the fixed point set of the identity component $I(G_1)_0$ on $CS(Y_1)$ has positive dimension.
\end{claim}

\begin{proof}
	By Lemma \ref{lem:dimIG} (1), if $\dim\mathrm{span}_\R(G_1\cdot y_1)=1$, then
	$$\R\times\{o_Z\}=\mathrm{span}_\R(G_1\cdot y_1)\supset G_1\cdot y_1\supset\R\cdot y_1,$$
	which forces $G_1\cdot y_1=\R\times\{o_Z\}$, contradicting the choice of $(Y_1,y_1,G_1)$. Hence $\dim\mathrm{span}_\R(G_1\cdot y_1)\ge2$. Thus $Y_1$ splits off an $\R^2$-factor on which $I(G_1)$ acts trivially. Consequently $Y_1$ is isometric to $C(S^1*Z)$ and $I(G_1)$ fixes $S^1$ and acts effectively on $Z$. By \cite{BPS24}, $Z$ is homeomorphic to $S^2$. This yields $\dim I(G_1)=1$, and since $I(G_1)_0$ fixes some point of $Z$, its fixed point set on $CS(Y_1)$ has positive dimension.
\end{proof}

We fix an isotropy type $K$ that minimizes the number of components of $I(G)$ among the triples $(Y,y,G)\in\Omega(\tilde M,\Gamma)$ satisfying the following:

\customitemize{c}
	\item \label{itm:c1}$\dim I(G)=1$;
	\item \label{itm:c2} the fixed point set of $I(G)_0$ on $CS(Y)$ has positive dimension.
\end{enumerate}
Let $\Omega_K$ be the subcollection of $(Y,y,G)\in\Omega(\tilde M,\Gamma)$ satisfying (\ref{itm:c1}) and (\ref{itm:c2}) with $I(G)\cong K$. By Claim \ref{claim:iso-fixed}, $\Omega_K\neq\emptyset$.
\begin{rem}\label{rem:fixpt}
For $(Y,y,G)\in\Omega(\tilde M,\Gamma)$, in the isometric decomposition $Y\cong C(S^0*Z)$ of Lemma \ref{lem:dimIG}, condition (\ref{itm:c2}) is equivalent to the requirement that $I(G)_0$ fixes some point of $Z$. Under the hypotheses of Theorem \ref{thm:main-codim5}, condition (\ref{itm:c2}) must be strengthened to: the fixed point set of $I(G)_0$ on $CS(Y)$ has dimension greater than $n-5$.
\end{rem}

\begin{lem}\label{lem:gap}
	There exists $\eta>0$ such that if $(Y,y,G)\in\Omega_K$ and $(X,x,H)\in\Omega(\tilde M,\Gamma)$ satisfy
	\begin{itemize}
		\item $d_{GH}\bigl((Y,y,G),(X,x,H)\bigr)\le\eta$,
		\item $(Y,y,G)$ satisfies property $(P)$,
	\end{itemize}
	then $(X,x,H)\in\Omega_K$.
\end{lem}
\begin{proof}
	Argue by contradiction. Assume $(Y_i,y_i,G_i)\in\Omega_K$ and $(X_i,x_i,H_i)\in\Omega(\tilde M,\Gamma)$ satisfy the following.
	\begin{itemize}
		\item $d_{GH}\bigl((Y_i,y_i,G_i),(X_i,x_i,H_i)\bigr)\le\eta_i\to0$,
		\item $(Y_i,y_i,G_i)$ satisfies property $(P)$,
		\item $(X_i,x_i,H_i)\notin\Omega_K$.
	\end{itemize}
	Up to a subsequence, we assume $(X_i,x_i,H_i)$ and $(Y_i,y_i,G_i)$ both converge to $(Y_\infty,y_\infty,G_\infty)$ in the equivariant Gromov--Hausdorff sense.
	
	By facts (\ref{itm:P2}) and (\ref{itm:P3}),
	$$
	(CS(Y_i),I(G_i))\GH(CS(Y_\infty),I(G_\infty)),\qquad
	(CS(X_i),I(H_i))\GH(CS(Y_\infty),K'),
	$$
	where $K'$ is a closed subgroup of $I(G_\infty)$. And there are injective homomorphisms $\phi_i:I(G_i)\to I(G_\infty)$ and $\psi_i:I(H_i)\to K'$ which are $\epsilon_i$-Gromov--Hausdorff approximations for some $\epsilon_i\to0$.
	
	\textbf{Case 1:} $\dim I(G_\infty)=1$.
	
	In this case, $\phi_i(I(G_i))$ is a closed subgroup of $I(G_\infty)$ of the same dimension, and it is $\epsilon_i$-dense in $I(G_\infty)$. Hence, for large $i$, $\phi_i$ is surjective, so $I(G_\infty)\cong I(G_i)$. Since $(Y_i,y_i,G_i)\in\Omega_K$, we have $I(G_i)\cong K$, and hence $I(G_\infty)\cong K$. By Lemma \ref{lem:dimIG}, for each $i$ there is an isometry $CS(Y_i)\cong S^0*W_i$ under which $I(G_i)$ acts effectively on $W_i$ and trivially on $S^0$. Passing to a subsequence, we may assume $W_i\GH W_\infty$, so that $CS(Y_\infty)\cong S^0*W_\infty$. By Remark \ref{rem:fixpt}, $I(G_i)_0$ fixes $S^0$ and some point of $W_i$. Together with the convergence $I(G_i)_0\GH I(G_\infty)_0$, this implies that $I(G_\infty)_0$ fixes $S^0$ and some point of $W_\infty$ as well; hence $(Y_\infty,y_\infty,G_\infty)\in\Omega_K$.

For the sequence $(X_i,x_i,H_i)$, since $\psi_i$ is injective, $I(H_i)$ is a closed subgroup of $I(G_\infty)\cong K$. In particular, $1\le \dim I(H_i)\le \dim K=1$. So $\psi_i(I(H_i)_0)=I(G_\infty)_0$, which implies that $I(H_i)_0\GH K'_0=I(G_\infty)_0$. Together with Lemma \ref{lem:dimIG}, up to a subsequence, we may choose isometries from $CS(X_i)$ and $CS(Y_\infty)$ to $S^0*Z_i$ and $S^0*Z_\infty$ respectively, satisfying $Z_i\GH Z_\infty$, where $Z_i,Z_\infty$ are spherical $3$-manifolds and $\mathrm{RCD}(2,3)$-spaces, and $I(H_i)_0,I(G_\infty)_0$ act effectively on $Z_i,Z_\infty$ and fix their $S^0$-factor. By \cite[Theorem 1.11, Remark 1.12]{BPS24}, there is a homeomorphism $f_i\colon Z_i\to Z_\infty$ which is an $\epsilon_i'$-Gromov--Hausdorff approximation, with $\epsilon_i'\to0$. Composing $f_i$ with an isometry of $Z_\infty$ if necessary, we may assume that $(f_i,\psi_i)$ is an $\epsilon_i'$-equivariant Gromov--Hausdorff approximation. Hence the push-forward action $f_i\circ I(H_i)_0\circ f_i^{-1}$ on $Z_\infty$ is $\epsilon_i'$-close to the $I(G_\infty)_0$-action and therefore converges uniformly to it. By the previous paragraph and Remark \ref{rem:fixpt}, $I(G_\infty)_0$ fixes some point of $Z_\infty$. Lemma \ref{lem:S^1fixedpt} then shows that, for large $i$, the action $f_i\circ I(H_i)_0\circ f_i^{-1}$ has a fixed point on $Z_\infty$; equivalently, $I(H_i)_0$ fixes a point of $Z_i$. Hence $(X_i,x_i,H_i)$ satisfies conditions (\ref{itm:c1}) and (\ref{itm:c2}). By the minimality of $K$ and $I(H_i)\hookrightarrow K$, $I(H_i)\cong K$, a contradiction.

\textbf{Case 2:} $\dim I(G_\infty)=2$.

We have $\phi_i(I(G_i)_0)\subset I(G_\infty)_0\cong T^2$. Since the number of components $\#\pi_0(I(G_i))$ is constant (it equals $\#\pi_0(K)$), the subgroups $\phi_i(I(G_i)_0)$ converge to $I(G_\infty)_0$. As $I(G_i)_0$ fixes some point of $W_i$, its limit $I(G_\infty)_0$ fixes some point of $W_\infty$. However, by Lemma \ref{lem:T2-no-fixed-points}, an effective $T^2$-action on a connected $3$-manifold has no fixed point, a contradiction.
\end{proof}
Let $\Omega_K^{P}$ be the collection of elements of $\Omega_K$ satisfying property $(P)$; by (\ref{itm:P1}) it is nonempty. Define $\mathcal U:=B_{0.5\eta}(\Omega_K^P)$, where $\eta$ is as in Lemma \ref{lem:gap}. Together, Lemma \ref{lem:gap} and (\ref{itm:P1}) show that condition (2) of Corollary \ref{cor:critical-scaling} applies to $\mathcal U$. Hence, for every $B\in\Omega(\tilde M,\Gamma)$, there exists $a\in[1,+\infty)$ such that $a^{-1}B\in\overline{\mathcal U}$. Applying Lemma \ref{lem:gap} once more, we obtain $a^{-1}B\in\Omega_K$, hence $B\in\Omega_K$. That is, $\Omega(\tilde M,\Gamma)=\Omega_K$. Propositions \ref{prop:isotropy-criteria} and \ref{prop:Pimpliesconnectedorbit} then imply that for some $m\in[0,n]$, for every $(Y,y,G)\in\Omega(\tilde M,\Gamma)$, $G\cdot y$ is isometric to $\R^m$. This contradicts the existence of $(Y_1,y_1,G_1)$.

Now the proof of Theorem \ref{thm:free-abelian-fg} is complete.

\subsection{$\pi_1(M)$ is torsion}\label{sec:torsional}
In this section, we prove the following theorem:
\begin{thm}\label{thm:torsion}
	Let $M$ be as in Theorem \ref{thm:main}. Further assume $\pi_1(M)$ is abelian and torsion. Then the same conclusion as Theorem \ref{thm:free-abelian-fg} holds. In particular, $\pi_1(M)$ is finite.
\end{thm}

We need the following lemma from Huang-Huang \cite[Lemma 3.2]{HH25a}.
\begin{lem}\label{lem:noncodim3}
	Suppose $(M^n_i,p_i)\GH(\R^{n-3}\times C(Z),(0^{n-3},o_Z))$, where $\Ric_{M_i}\ge0$ and $\vol B_1(p_i)\ge v>0$. If $(\hat M_i,\hat p_i)\to (M_i,p_i)$ are normal covers with deck transformation groups $F_i$ and $\diam(F_i\cdot \hat p_i)\to0$, then $F_i$ is trivial for all sufficiently large $i$.
\end{lem}

Now we begin to prove Theorem \ref{thm:torsion}.

\begin{lem}\label{lem:1dimspanG}
	For any $(Y,y,G) \in \Omega(\tilde M,\Gamma)$, the orbit $G\cdot y$ spans at most a $1$-dimensional space.
\end{lem}
\begin{proof}
	Suppose not. Fix a non-trivial $\gamma \in \Gamma$. Assume that for $r_i \to 0$, we have the following convergence,
	$$(r_i\tilde M,\tilde  p, \Gamma, \spa{\gamma}) \xrightarrow{GH} (Y, y, G, H).$$
	Then $(Y,y)\cong(\R^k \times C(Z),(0^k,o_Z))$, where $\R^k \times \{o_Z\} = \mathrm{span}_\R(G\cdot y)$, with $k \ge 2$. Since $G$ is abelian and $H$ fixes $y$, $H$ fixes $Gy$, and hence fixes $\mathrm{span}_\R(G\cdot y)$. So $H$ acts on the $\R^k$-factor of $Y$ trivially. By Fukaya-Yamaguchi \cite{FY92}, up to a subsequence,
	$$(r_i\tilde M/\spa{\gamma},[\tilde p]) \xrightarrow{GH} (Y/H,[y]).$$
	By the above discussion, $Y/H \cong \R^k \times \left(C(Z)/H\right)$. Note that the above convergence is noncollapsed. We can apply Lemma \ref{lem:noncodim3} to the above sequence, which yields that $\spa{\gamma}$ is trivial. This is a contradiction to the choice of $\gamma$.
\end{proof}
Therefore, to prove Theorem \ref{thm:torsion}, it suffices to show that every orbit $G\cdot y$ is either a single point or isometric to $\R$. We argue by contradiction in the remainder of the proof, assuming that this is not the case.

\begin{claim}\label{claim:orbit-type}
	For any $(Y,y,G)\in\Omega(\tilde M,\Gamma)$, one of the following holds:
	\begin{itemize}
		\item $G\cdot y=\{y\}$;
		\item $(Y,y)$ splits isometrically as $(\R\times C(Z),(0,o_Z))$, where $Z$ is a spherical $3$-manifold; moreover, the action of $G$ respects this splitting, acting on the $\R$-factor by reflection about a nonzero point and fixing $o_Z$.
	\end{itemize}
	Both types are realized in $\Omega(\tilde M,\Gamma)$.
\end{claim}

\begin{proof}
	By Lemma \ref{lem:1dimspanG}, the orbit $G\cdot y$ is of one of the four types in Lemma \ref{lem:1dim}. Let $S$ be the collection of those $(Y,y,G)\in\Omega(\tilde M,\Gamma)$ whose orbit is of type (4), and suppose that $S\neq\emptyset$. If $(X,x,H)\in\Omega(\tilde M,\Gamma)$ satisfies $d_{GH}((X,x,H),(Y,y,G))<0.01$ for some $(Y,y,G)\in S$, then $(X,x,H)$ is of type (3) or (4). By Corollary \ref{cor:critical-scaling} (1) applied to $\mathcal U:=B_{0.005}(S)$, every element of $\Omega(\tilde M,\Gamma)$ has orbit of type (3) or (4). Since blowing up type (3) yields type (1), only type (4) can occur, contradicting that some orbit $G\cdot y$ is neither a single point nor isometric to $\R$. Hence $S=\emptyset$ and only types (1) and (2) can occur. And by \cite{BPS24}, $Z$ is homeomorphic to a spherical $3$-manifold.
\end{proof}

Since $\tilde M$ is simply connected, it is orientable. Up to passing to an index-$2$ subgroup of $\Gamma$ if necessary at the beginning of the proof, we may assume that every $\gamma\in\Gamma$ is orientation preserving. By Lemma \ref{lem:orientation}, every asymptotic cone at infinity of $\tilde M$ is orientable, and every element of any limit group $G$ is orientation preserving (for the orientability of Ricci limit spaces, see Definition \ref{def:orientability}).

By Claim \ref{claim:orbit-type}, there exists $(Y,y,G)\in\Omega(\tilde M,\Gamma)$ with $G\cdot y$ disconnected, which has an isometric splitting
\begin{equation}\label{eq:splitting}
	(Y,y)\cong(\R\times C(Z),(0,o_Z))
\end{equation}
in which $G$ acts on the $\R$-factor by reflection and $Z$ is a spherical $3$-manifold. Fix $\rho\in G$ with $\rho y\neq y$. By the preceding paragraph, $\rho$ is orientation preserving; together with Lemma \ref{lem:orientation-on-open-subsets}, this shows that its restriction to $Y\setminus(\R\times\{o_Z\})$, which is homeomorphic to $\R\times\R\times Z$, is orientation preserving. Since $\rho$ reverses the orientation of the first factor and preserves that of the second, it must reverse the orientation of $Z$.

By Propositions \ref{prop:isotropy-criteria} and \ref{prop:Pimpliesconnectedorbit}, $\dim I(G)_0\ge1$. As $I(G)$ commutes with $\rho$, it fixes both $y$ and $\rho y$, hence the whole line through them. Thus $I(G)$ acts effectively on $Z$. Applying Lemma \ref{lem:centralizers}, we obtain $\dim I(G)_0=1$ and a point $z_0\in Z$ fixed by $I(G)_0$. Consequently,
\begin{equation}\label{itm:F}
	\text{$I(G)_0$ fixes the $\R$-factor of the splitting \eqref{eq:splitting} and the ray $\{(0,t,z_0)\mid t\ge0\}$.}
\end{equation}
We fix a convergence $(r_i\tilde M,\tilde p,\Gamma)\GH(Y,y,G)$ for some $r_i\to0$.
\begin{lem}\label{lem:torsion-embedding}
	There exists an injective homomorphism $\iota\colon\Gamma\to I(G)$ such that, passing to a subsequence, for every $\gamma\in\Gamma$,
	$$(r_i\tilde M,\tilde p,\gamma,\spa{\gamma})\GH(Y,y,\iota(\gamma),\spa{\iota(\gamma)}).$$
\end{lem}

\begin{proof}
	Fix $\gamma\in\Gamma$. Passing to a subsequence, we have $(r_i\tilde M,\tilde p,\gamma,\spa{\gamma})\GH(Y,y,\gamma_\infty,K)$, and $K<I(G)$ since $\gamma$ is torsion. 
	
	Every $g\in K$ is a limit of powers of $\gamma$: passing to a further subsequence, there exist $j_i\in[0,\mathrm{ord}(\gamma))$ with $\gamma^{j_i}\GH g$. Since the $j_i$ take values in a finite set, we may assume $j_i\equiv j$ for some $j\in[0,\mathrm{ord}(\gamma))$, and then $g=\gamma_\infty^j\in\spa{\gamma_\infty}$. Thus $K\subset\spa{\gamma_\infty}$, and $\#K\le\mathrm{ord}(\gamma_\infty)\le\mathrm{ord}(\gamma)$.

	Applying \cite[Corollary 2.2]{PR18}, we obtain $\mathrm{ord}(\gamma)=\#\spa{\gamma}\le\#K$. Together with the inequality above, this yields $K=\spa{\gamma_\infty}$ and $\mathrm{ord}(\gamma)=\mathrm{ord}(\gamma_\infty)$.
	
	Since $\Gamma$ is countable, a standard diagonal argument gives a map $\iota:\gamma\mapsto\gamma_\infty$. It is obvious that $\iota$ is an injective homomorphism. This completes the proof.
\end{proof}

 Since $\Gamma$ is infinite and $I(G)$ has finitely many connected components, there exists a nontrivial $\gamma\in\Gamma$ with $\gamma_\infty:=\iota(\gamma)\in I(G)_0$, and
 \begin{equation}\label{eq:conv}
 	(r_i\tilde M,\tilde p,\spa\gamma)\GH(Y,y,\spa{\gamma_\infty}).
 \end{equation}
 By (\ref{itm:F}) and the discussion above, $\spa{\gamma_\infty}$ fixes every point of the set $\{(s,t,z_0)\in Y \mid s\in\R,\ t\ge0\}$. Under the convergence \eqref{eq:conv}, choose $x_i\in r_i\tilde M$ with $x_i\to y':=(0,1,z_0)$, and note that $\sigma_i:=\diam_{r_i\tilde M}(\spa\gamma\cdot x_i)\to\diam(\spa{\gamma_\infty}\cdot y')=0$. Take $R_i\to\infty$ slowly enough that $R_i\sigma_i\to0$ and $R_ir_i\to0$, and
 $$(R_ir_i\tilde M,x_i,\spa\gamma)\GH(T_{y'}Y,y'_\infty,\spa{\gamma'_\infty}),$$
 where $T_{y'}Y$, a tangent cone of $Y$ at $y'$, splits off an $\R^2$-factor, and $\gamma'_\infty$ is a limit of $\gamma_\infty$. Moreover, $\spa{\gamma'_\infty}$ fixes this $\R^2$-factor pointwise, so the quotient $T_{y'}Y/\spa{\gamma'_\infty}$ is again a metric cone splitting off an $\R^2$-factor. Note that $\tilde M/\spa{\gamma}$ still has Euclidean volume growth. Hence we can apply Lemma \ref{lem:noncodim3} to the diagram
 \begin{equation*}\label{dia:00}
 	\xymatrix@C=2.5cm{
 		(R_ir_i\tilde M,x_i,\spa\gamma)\ar[d]\ar[r]^{GH} & (T_{y'}Y,y'_\infty,\spa{\gamma'_\infty})\ar[d]\\
 		(R_ir_i\tilde M,x_i)/\spa\gamma\ar[r]^{GH} & (T_{y'}Y,y'_\infty)/\spa{\gamma'_\infty},
 	}
 \end{equation*}
 and conclude that $\gamma$ is trivial, contradicting the choice of $\gamma$. This completes the proof of Theorem \ref{thm:torsion}.

\section{Virtual abelianness, poles, and escape rate}

In this section, we prove Theorem \ref{thm:main-codim5} and Theorem \ref{thm:main-codim6}. 

Since the proof of virtual abelianness in Theorem \ref{thm:main-codim5} is analogous to that in Theorem \ref{thm:main-codim6}, we prove the latter. 
\subsection*{Proof of Theorem \ref{thm:main-codim6}}

We first need the following lemma, which is \cite[Proposition 7.4]{IsPa65}.
\begin{lem}\label{lem:virtuallyabeliangroups}
	Let $G$ be a group and let $m$ be a positive integer. Then $G$ contains an abelian subgroup of index at most $m$ if and only if every finitely generated subgroup of $G$ contains an abelian subgroup of index at most $m$.
\end{lem}

	By Lemma \ref{lem:virtuallyabeliangroups}, it suffices to show that there exists $C(n,\theta)>0$ such that every finitely generated subgroup of $\pi_1(M)$ contains an abelian subgroup of index at most $C(n,\theta)$. We assume that $H<\pi_1(M)$ is finitely generated. Since all the hypotheses are preserved under passage to covers of $M$, by \cite{KW11}, $H$ contains a nilpotent subgroup $\Gamma$ of index bounded by some $C(n)$, with nilpotency length at most $n$. In what follows, put $N:=\tilde M/\Gamma$.
	
	For arbitrary $r_i\to0$, up to a subsequence, we have the following commutative diagram, 
	\begin{equation}\label{dia:01}
		\xymatrix@C=2.5cm{
			(r_i\tilde M,\tilde p,\Gamma)\ar[d]\ar[r]^{GH} & (Y,y,G)\ar[d]\\
			(r_iN,p)\ar[r]^{GH} & (X,x).
		}
	\end{equation}
	By assumption, $X$ splits off an $\R^{n-6}$-factor. Lifting it to $Y$ and using that $\tilde M$ has Euclidean volume growth, we obtain $(Y,y)\cong(\R^{n-6}\times\R^b\times C(Z),(0^{n-6},0^b,o_Z))$, where $\diam Z<\pi$ and $G$ acts trivially on the first $\R^{n-6}$-factor. Clearly $b\in\{0,1,\ldots,6\}$. If $b\ge5$, then Cheeger--Colding's codimension-$2$ regularity \cite{CC97} and Colding's rigidity \cite{Col97} force $M$ to be flat; hence $M$ is a product of $\R^{n-6}$ with a flat $6$-manifold, and the conclusion follows from Bieberbach's theorem. Thus we may assume $b\le4$. In particular, $G\cdot y$ is a closed submanifold of $\R^b$ with dimension $\le 4$. 
	
\textbf{Case 1:} $\dim G\cdot y\le 3$ for every $(Y,y,G)\in\Omega(\tilde M,\Gamma)$. Note that $G\cdot y\subset\{0^{n-6}\}\times\R^b\times\{o_Z\}$. By \cite[Proposition 2.2]{Ye24}, for some $C>0$ and all sufficiently large $R$, $\#\Gamma(R)\le CR^{3.1}$, where $\Gamma(R):=\{\gamma\in\Gamma\mid d(\tilde p,\gamma\tilde p)\le R\}$. Fix a finite symmetric generating set $S\subset\Gamma$, and put $L:=\max\{d(\tilde p,\gamma\tilde p)\mid\gamma\in S\}$. For $\gamma\in\Gamma$, denote by $\|\gamma\|_S$ its word length with respect to $S$, that is, the least integer $m\ge0$ such that $\gamma=s_1\cdots s_m$ for some $s_1,\ldots,s_m\in S$. Then $d(\tilde p,\gamma\tilde p)\le L\|\gamma\|_S$ for every $\gamma\in\Gamma$. Let $\beta_\Gamma(R):=\#\{\gamma\in\Gamma\mid \|\gamma\|_S\le R\}$. For sufficiently large $R$, we have $\{\gamma\in\Gamma\mid \|\gamma\|_S\le R\}\subset\Gamma(LR)$, and hence
$$\beta_\Gamma(R)\le\#\Gamma(LR)\le C(L)R^{3.1}.$$
Hence Proposition \ref{prop:poly3} shows that $\Gamma$ is virtually $\Z^k$ for $k\in\{0,1,2,3\}$.

\textbf{Case 2:} $\dim G_1\cdot y_1=4$ for some $(Y_1,y_1,G_1)\in\Omega(\tilde M,\Gamma)$. This forces $b=4$, so $G_1\cdot y_1$ is isometric to $\R^4$. \cite[Lemma 2.7]{HH25b} then implies that $G\cdot y$ is isometric to $\R^4$ for every $(Y,y,G)\in\Omega(\tilde M,\Gamma)$. Then \cite[Theorem 2.8]{HH25b} implies that $\Gamma$ is virtually $\Z^4$.

Hence we have proved that $\Gamma$ is virtually $\Z^k$ for some $0\le k\le6$. A finitely generated virtually abelian nilpotent group has finite commutator subgroup. By \cite[Lemma 4.5]{Pan24} and \cite[Theorem 5.1]{Pan25}, the index of the center $Z(\Gamma)$ of $\Gamma$ is bounded by a constant $C_1(n,\theta)$. Hence $H$ contains an abelian subgroup $Z(\Gamma)$ of index at most $C(n)C_1(n,\theta)$. As $H$ was arbitrary, the required conclusion follows.

\begin{rem}
	The hypothesis of \cite[Lemma 2.7]{HH25b} with $k=4$ requires the maximal Euclidean factor of every asymptotic cone to have dimension at most four. Here the cone in Case 2 has maximal Euclidean factor $\R^{n-6}\times\R^4$, whose dimension exceeds four when $n>6$. However, $G$ acts trivially on the $\R^{n-6}$-factor, and every central orbit lies in the $G$-invariant slice $\{0^{n-6}\}\times\R^b\times\{o_Z\}$, where $b\le4$. The Euclidean orbit and critical scaling arguments in the proof of that lemma apply to these slices and give the same conclusion.
\end{rem}

\subsection*{Proof of Theorem \ref{thm:main-codim5} (2) \& (3)}

By Theorem \ref{thm:main-codim5} (1), $\pi_1(M)$ is finitely generated and virtually abelian. Let $\Gamma$ be a normal abelian subgroup of $\pi_1(M)$ of finite index; then $\Gamma\cong H\oplus T$, where $H\cong\Z^k$ is free abelian and $T$ is its torsion subgroup.

For any $r_i\to0$, up to a subsequence, consider the equivariant convergence
$$(r_i\tilde M,\tilde p,H,\Gamma,\pi_1(M))\GH(Y,y,H_\infty,\Gamma_\infty,G).$$

Applying Theorem \ref{thm:free-abelian-fg} and Remark \ref{rem:codim5} to $\tilde M/H$, we conclude that $H_\infty\cdot y$ is isometric to $\R^l$, so $Y/H_\infty$ has a pole at $yH_\infty$. Since $\Gamma/H$ and $\pi_1(M)/\Gamma$ are both finite, the quotient groups $\Gamma_\infty/H_\infty$ and $G/\Gamma_\infty$ fix $yH_\infty$ and $y\Gamma_\infty$, respectively; hence $Y/G$ has a pole at $yG$. As $r_i\to0$ was arbitrary, $(Y/G,yG)$ runs through $\Omega(M)$, and each such cone has a pole at its reference point; this proves (2). Finally, conclusion (3) follows from \cite[Proposition 4.2]{Pan22}.

\section{Appendix}

\subsection{Euclidean orbits contained in a line}
\begin{lem}\label{lem:1dim}
	Let $G$ be a closed abelian subgroup of the Euclidean isometry group $\Isom(\R^m)$. Suppose that the orbit of the origin $G\cdot o=\{g\cdot o \mid g\in G\}$ is contained in a one-dimensional linear subspace $L\subset\R^m$. Then $G\cdot o$ is one of the following four types:
	\begin{itemize}
		\item [(1)] $G\cdot o=\{o\}$;
		\item [(2)] $G\cdot o=\{o,t\}$ for some nonzero $t\in L$;
		\item [(3)] $G\cdot o=\Z a=\{na \mid n\in\Z\}$ for some nonzero $a\in L$;
		\item [(4)] $G\cdot o=L$.
	\end{itemize}
	In cases (2)--(4), the $G$-action splits on $L\oplus L^{\perp}$, and $G|_L$ acts by translations or reflections.
\end{lem}

\begin{proof}
	The orbit $G\cdot o$ is homeomorphic to the homogeneous space $G/G_o$, which is a manifold of dimension $\dim G-\dim G_o$. Since $G\cdot o\subset L$ and $\dim L=1$, we have $\dim(G/G_o)\le 1$.
	
	Case 1: $\dim(G/G_o)=1$.
	Then $G\cdot o$ is a one-dimensional embedded submanifold of the one-dimensional manifold $L$. It is also closed in $L$, because orbits of closed groups of isometries are closed. Hence $G\cdot o$ is a nonempty open and closed subset of the connected space $L$, so $G\cdot o=L$.
	
	Case 2: $\dim(G/G_o)=0$.
	Then $G/G_o$ is discrete, so $G\cdot o$ is a discrete closed subset of $L$. The identity component $G_0$ of $G$ is contained in $G_o$.
	
	If $G\cdot o=\{o\}$, we are done. Assume from now on that $G\cdot o\neq\{o\}$. Since $G\cdot o$ is discrete and nonzero, there exists a vector $v\in G\cdot o$ of minimal positive norm. Choose $g=(A,v)\in G$. Then $Av=v$ or $Av=-v$.
	
	Subcase 1: $Av=v$.
	Then for any $t\in\R$, $g\cdot tv=(t+1)v$. So $g$ acts on $L$ as a translation by $v$. By the minimality of $\|v\|$, $G\cdot o=\Z v$.
	
	Subcase 2: $Av=-v$.
	Then for any $t\in\R$, $g\cdot tv=(1-t)v$. That is, $g$ is a reflection about $0.5v$. If there exists $h\in G$ which moves $0.5v$, then $g$ fixes $2$ distinct points of $L$ hence all of $L$, a contradiction. Hence $G\cdot o=\{o,v\}$.
\end{proof}

\subsection{Fixed point lemmas on $3$-manifolds}

All manifolds in this subsection are assumed to be closed and connected.

\begin{lem}\label{lem:S^1fixedpt}
	Let $M$ be a closed topological $3$-manifold, and let $\phi_i,\phi:S^1\times M\to M$ be effective continuous actions. Suppose that $\phi_i\to\phi$ uniformly. If $$M_{\phi}^{S^1}:=\{x\in M\mid\phi(g,x)=x\text{ for every }g\in S^1\}\neq\varnothing,$$ then $M_{\phi_i}^{S^1}\neq\varnothing$ for all sufficiently large $i$.
\end{lem}

\begin{proof}
	By Raymond's smoothing theorem \cite[Theorem 6]{Raymond68}, each action is smoothable. By \cite[Chapter VI, Theorems 4.1 and 2.2]{Bredon72}, choose a finite-dimensional orthogonal representation $\rho:S^1\to\mathrm{O}(V)$, an equivariant embedding $e:(M,\phi)\hookrightarrow V$, and an equivariant retraction $r:U\to e(M)$ from an invariant open neighborhood.
	
	Following \cite[Section 3]{Pal61}, define
	$$A_i(x):=\int_{S^1}\rho(g^{-1})e(\phi_i(g,x))\,\mathrm{d}g,$$
	where $\mathrm{d}g$ is normalized Haar measure. Then $A_i\to e$ uniformly, and Haar invariance gives $A_i(\phi_i(h,x))=\rho(h)A_i(x)$. For large $i$, the map $\psi_i:=e^{-1}\circ r\circ A_i$ is therefore well defined and equivariant from $(M,\phi_i)$ to $(M,\phi)$. Compactness and uniform convergence ensure that every segment joining $e(x)$ to $A_i(x)$ lies in $U$. Applying $e^{-1}\circ r$ to these segments gives $\psi_i\simeq\id_M$.
	
	Fix such an $i$, and suppose that $M_{\phi_i}^{S^1}=\varnothing$. All isotropys of $\phi_i$ are then finite, and their orders take only finitely many values by \cite[Chapter IV, Proposition 1.2]{Bredon72}. Choose an odd prime $p$ larger than all these orders, and let $C_p<S^1$ have order $p$. Since $C_p$ cannot lie in any stabilizer of $\phi_i$, we have
	\begin{equation}\label{eq:emptyfix}
			M_{\phi_i}^{C_p}=\varnothing,\qquad M_\phi^{C_p}\supset M_\phi^{S^1}\neq\varnothing.
	\end{equation}

	Suppose first that $M$ is orientable. The map $\psi_i$ is $C_p$-equivariant and has degree $1$. Since $p$ is odd, \cite[Theorem 5]{HankePuppe06} implies that $\psi_i|_{M_{\phi_i}^{C_p}}:M_{\phi_i}^{C_p}\to M_\phi^{C_p}$ is surjective, contradicting (\ref{eq:emptyfix}).
	
	Suppose now that $M$ is nonorientable. Let $\pi:\hat M\to M$ be the orientation double cover. The actions on local orientations give canonical lifts $\hat\phi_i$ and $\hat\phi$. Lifting the homotopy from $\id_M$ to $\psi_i$ gives a map $\hat\psi_i:\hat M\to\hat M$ homotopic to $\id_{\hat M}$, so $\deg\hat\psi_i=1$.

		The map $\hat\psi_i$ is equivariant. Indeed, for each $h\in S^1$, the maps $\hat\psi_i\circ\hat\phi_i(h,\cdot)$ and $\hat\phi(h,\cdot)\circ\hat\psi_i$ are lifts of the same map $\psi_i\circ\phi_i(h,\cdot)=\phi(h,\cdot)\circ\psi_i$ on $M$, so they differ by a deck transformation. This deck transformation depends continuously on $h$ and is the identity at $h=e$. Since $S^1$ is connected, it is always the identity. Moreover, since $p$ is odd, a point of $M$ is $C_p$-fixed for $\phi$ exactly when both of its lifts in $\hat M$ are. It follows that $\hat M_{\hat\phi}^{C_p}\neq\emptyset$. And it is obvious that $\hat M_{\hat\phi_i}^{C_p}=\varnothing$. Applying \cite[Theorem 5]{HankePuppe06} to $\hat\psi_i$ gives the same contradiction.
	
	Therefore $M_{\phi_i}^{S^1}\neq\varnothing$ for all sufficiently large $i$.
\end{proof}

\begin{lem}\label{lem:T2-no-fixed-points}
	Let $T^2$ act continuously and effectively on a topological $3$-manifold $M$. Then $M^{T^2}=\varnothing$.
\end{lem}

\begin{proof}
	By \cite[Proposition~5]{Raymond68}, the action is equivariantly smoothable. Thus we may assume that it is smooth. Suppose that $p\in M^{T^2}$. After averaging a Riemannian metric over $T^2$, the action is isometric. The isotropy representation $\iota_p\colon T^2\longrightarrow O(T_pM)$ is faithful. Indeed, an element in its kernel is an isometry fixing $p$ with identity differential, hence is the identity on the connected manifold $M$; effectiveness then makes it the identity of $T^2$.
	
	Since $T^2$ is connected, the image of $\iota_p$ lies in $SO(T_pM)\cong SO(3)$. This is impossible because $\operatorname{rank}SO(3)=1$. Therefore $M^{T^2}=\varnothing$.
\end{proof}

\begin{lem}\label{lem:centralizers}
	Let $M$ be a closed spherical topological $3$-manifold. Let $G$ be a compact connected abelian Lie group acting continuously and effectively on $M$. Suppose that an orientation reversing homeomorphism $\rho:M\to M$ centralizes the $G$-action, that is, $\rho(gx)=g\rho(x)$ for all $g\in G$ and $x\in M$. Then $\dim G\le1$. Moreover, if $G\cong S^1$, then $M^G\neq\varnothing$.
\end{lem}
\begin{proof}
	Since $G$ is compact, connected, and abelian, $G\cong T^k$ for some $k\ge0$. Since $M$ is spherical, $\pi_1(M)$ is finite and $M$ is orientable. Hence $H_1(M;\Q)=0$, and Poincaré duality gives $H_2(M;\Q)=0$. Writing $L$ for the rational Lefschetz number, we have $L(\rho)=1-(-1)=2$.
	
	If $k=0$, there is nothing to prove. Suppose that $k\ge1$. By \cite[Proposition 5]{Raymond68}, the action is smoothable. Hence $M$ is a compact $G$-ENR (equivariant Euclidean neighborhood retract) by \cite[Chapter VI, Theorems 4.1 and 2.2]{Bredon72}. By \cite[Theorem~1(iv)]{Kom87}, $L\bigl(\rho|_{M^G}\bigr)=L(\rho)=2$, and hence $M^G\neq\varnothing$.
	
	If $k\ge2$, choose a subtorus $K\cong T^2\subset G$. The restricted $K$-action is effective. Lemma~\ref{lem:T2-no-fixed-points} gives $M^K=\varnothing$, contradicting $M^G\subset M^K$. Therefore $k\le1$. 
\end{proof}

\subsection{Orientability and orientation preserving maps}\label{subsec:orientability-of-maps}

We follow \cite{BrBrPi24} for the definition of orientability.

\begin{defn}\label{def:orientability}
	Let $(X,d,\mathcal H^n)$ be a noncollapsed $\mathrm{RCD}(\kappa,n)$ space with no boundary, and set
	$$X_{\mathcal M}:=\{x\in X\mid x\text{ has an open neighborhood homeomorphic to }\R^n\}.$$
	\begin{itemize}
		\item [(1)] $X$ is called orientable if $X_{\mathcal M}$ is orientable.
		\item [(2)] An orientation of $X_{\mathcal M}$ is called an orientation of $X$.
		\item [(3)] Fix an orientation of $X$. A homeomorphism $f:X\to X$ is called orientation preserving if $f|_{X_{\mathcal M}}$ preserves orientation.
	\end{itemize}
\end{defn}

Item (1) agrees with \cite[Definition 1.6]{BrBrPi24}; see also \cite[Proposition 1.7 and Theorem 2.1]{BrBrPi24}.

\begin{lem}\label{lem:orientation-on-open-subsets}
	Let $X$ be oriented as in Definition \ref{def:orientability}, and let $f:X\to X$ be a homeomorphism. For every open subset $A\subseteq X$ that is a topological $n$-manifold, give $A$ and $f(A)$ the orientations induced from $X_{\mathcal M}$. Then $f$ preserves orientation if and only if $f|_A:A\to f(A)$ preserves orientation for every such $A$. It suffices to check this on one nonempty open subset of $X_{\mathcal M}$.
\end{lem}

\begin{proof}
	Every such $A$ lies in $X_{\mathcal M}$, and
	$f(X_{\mathcal M})=X_{\mathcal M}$. Thus the first assertion follows
	by restriction, taking $A=X_{\mathcal M}$ for the converse.
	
	For sufficiently small $\varepsilon>0$, the set $A_\varepsilon(X)$
	is connected and dense in $X$; see
	\cite[p.~3, following~(1.3)]{BrBrPi24}.
	Since $A_\varepsilon(X)\subset X_{\mathcal M}$, the latter is connected.
	The orientation sign of $f|_{X_{\mathcal M}}$ is locally constant,
	hence constant, proving the last assertion.
\end{proof}

\begin{lem}\label{lem:orientation}
	Let
	$$
	(X_i,d_i,\mathcal H^n,p_i,f_i)\GH(X,d,\mathcal H^n,p,f)
	$$
	be an equivariant pointed Gromov--Hausdorff convergent sequence of
	noncollapsed $\mathrm{RCD}(\kappa,n)$-spaces with no boundary, where
	$f_i\in\Isom(X_i)$ and $f\in\Isom(X)$.
	If every $X_i$ is orientable and every $f_i$ preserves orientation,
	then $X$ is orientable and $f$ preserves orientation.
\end{lem}

\begin{proof}
	By \cite[Theorem~1.16]{BrBrPi24}, $X$ is orientable.
	Fix an orientation of $X$. For small $\varepsilon>0$,
	$A:=A_\varepsilon(X)$ is a connected, $f$-invariant open subset
	of $X_{\mathcal M}$.
	
	Fix $x\in A$. The gluing argument in
	\cite[proof of Theorem~4.1, Proof~I]{BrBrPi24}
	gives a connected open set $U\Subset A$ containing $x$ and $f(x)$,
	and open embeddings $h_i:U\longrightarrow (X_i)_{\mathcal M}$ approximating the Gromov--Hausdorff convergence.
	Since $U$ is connected, we may orient each $X_i$ so that $h_i$
	preserves orientation.
	
	Choose a small connected neighborhood $V\Subset U$ of $x$ such that
	$f(\overline V)$ lies in an oriented coordinate chart $W\subset U$.
	For all large $i$, the maps $g_i:=h_i^{-1}\circ f_i\circ h_i:V\longrightarrow W$ are well defined, preserve orientation, and converge locally uniformly
	to $f|_V$.
	Applying Sublemma~\ref{lem:stability-of-orientation-sign} below in the chart $W$, we conclude that $f|_V$ preserves orientation.
	Lemma~\ref{lem:orientation-on-open-subsets} completes the proof.
	
	\begin{slem}\label{lem:stability-of-orientation-sign}
		Let $U$ be a connected oriented topological $n$-manifold without
		boundary, and let $f_i,f:U\to\R^n$ be open embeddings.
		If $f_i\to f$ locally uniformly, then their orientation signs agree
		for all sufficiently large $i$.
	\end{slem}
	
	\begin{proof}
		Fix $x\in U$. Replacing $f_i$ and $f$ by $f_i-f_i(x)$ and $f-f(x)$
		preserves their orientation signs and locally uniform convergence.
		Thus we may assume $f_i(x)=f(x)=0^n$.
		Choose $r>0$ with $\overline{B_r(0^n)}\subset f(U)$, and put
		$D:=f^{-1}(\overline{B_r(0^n)})$.
		For large $i$, we have $\sup_D|f_i-f|<r/2$.
		
		Define the continuous homotopy
		$$
		H_i:[0,1]\times D\to\R^n,\qquad
		H_i(t,y):=(1-t)f(y)+tf_i(y).
		$$
		Since $|f(y)|=r$ for $y\in\partial D$, we have
		$$
		|H_i(t,y)|
		\ge r-t|f_i(y)-f(y)|>r/2
		\qquad (t\in[0,1],\ y\in\partial D).
		$$
		Thus no point of $\partial D$ is mapped to $0^n$ at any time.
		In particular, every $H_i(t,\cdot)$ is a map of pairs $(D,\partial D)\longrightarrow
		(\R^n,\R^n\setminus\{0^n\})$. Since $H_i(0,\cdot)=f|_D$ and $H_i(1,\cdot)=f_i|_D$,
		homotopy invariance gives the same induced map
		$$
		H_n(D,\partial D;\Z)\longrightarrow
		H_n(\R^n,\R^n\setminus\{0^n\};\Z).
		$$
		
		Both embeddings have $x$ as their unique preimage of $0^n$.
		The natural maps
		$$
		H_n(D,\partial D;\Z)
		\longrightarrow H_n(D,D\setminus\{x\};\Z)
		\longrightarrow H_n(U,U\setminus\{x\};\Z)
		$$
		are isomorphisms, by radial deformation retraction in the chart $f$
		and excision, respectively.
		Hence $f_i$ and $f$ induce the same map on local homology at $x$.
		Their orientation signs therefore agree at $x$, and thus on $U$
		by connectedness.
	\end{proof}
\end{proof}

\subsection{Nilpotent groups of growth degree strictly less than four}

The following proposition follows from \cite[Theorem 3.10 and Proposition 3.11]{DHL17}.

\begin{prop}\label{prop:poly3}
	Let $\Gamma$ be a finitely generated nilpotent group whose polynomial growth degree is less than $4$. Then $\Gamma$ is virtually $\Z^k$ for some $0\le k\le3$.
\end{prop}

\begin{proof}
	If $\Gamma$ is finite, take $k=0$. Otherwise, its growth degree $d$ is a positive integer by \cite[Theorem 3.10]{DHL17}, so $d\in\{1,2,3\}$. By \cite[Proposition 3.11]{DHL17}, $\Gamma$ contains a finite index free abelian subgroup of rank $d$.
\end{proof}

\bibliographystyle{alpha}
\bibliography{ref}

\end{document}